\documentclass[10pt]{amsart}
\usepackage{amsthm}

\usepackage{mathtools}

\usepackage{graphicx}
\usepackage{amssymb}
\usepackage{epstopdf}
\usepackage{hyperref}
\usepackage{xypic}

\newcommand\CC{\mathbb{C}}

\newcommand\QQ{\mathbb{Q}}
\newcommand\RR{\mathbb{R}}

\newcommand\TT{\mathbb{T}}

\newcommand\ZZ{\mathbb{Z}}

\newcommand\cI{\mathcal{I}}

\newcommand\cLL{\mathcal{L}}

\newcommand\cP{\mathcal{P}}

\newcommand\cT{\mathcal{T}}

\newcommand\Sp{\text{Sp}}

\newtheorem{theorem}{Theorem}[section]
\newtheorem{lemma}[theorem]{Lemma}
\newtheorem{proposition}{Proposition}
\newtheorem{corollary}{Corollary}
\newtheorem{lemma*}{Lemma}
\theoremstyle{definition}

\newtheorem{notation*}{Notation}

\newtheorem{question}{Question}

\title{Finiteness Properties of the Johnson Subgroups}
\author{Kevin Kordek}
\date{}
\begin{document}
\maketitle
\begin{abstract}
The main goal of this note is to provide evidence that the first rational homology of the Johnson subgroup $K_{g,1}$ of the mapping class group of a genus g surface with one marked point is finite-dimensional. Building on work of Dimca-Papadima \cite{dimcapapa12}, we use symplectic representation theory to show that, for all $g > 3$, the completion of $H_1(K_{g,1},\mathbb{Q})$ with respect to the augmentation ideal in the rational group algebra of $\mathbb{Z}^{2g}$ is finite-dimensional. We also show that the terms of the Johnson filtration of the mapping class group have infinite-dimensional rational homology in some degrees in almost all genera, generalizing a result of Akita.
\end{abstract}
\section{Introduction}
\indent Let $\pi$ denote the fundamental group of a closed orientable surface $S_g$ of genus $g\geq 2$ and let $\Gamma_{g,r}$ denote the mapping class group of $S_g$ with $r$ marked points, where $r=0\ \text{or}\ 1$. 

The \emph{$n$th Johnson subgroup} $K_{g,1}(n)$ of $\Gamma_{g,1}$ and the \emph{$n$th outer Johnson subgroup $K_g(n)$} of $\Gamma_g:=\Gamma_{g,0}$ are defined, respectively, to be the kernels of the natural maps 
\begin{equation*}
\Gamma_{g,1}\rightarrow \text{Aut}(\pi/\pi^{(n+1)})\ \ \ \ \ \ \ \ \Gamma_g\rightarrow \text{Out}(\pi/\pi^{(n+1)}).
\end{equation*}
Here $\pi^{(n)}$ denotes the $n$th term of the lower central series of $\pi$. The classical Torelli groups $T_{g,1}$ and $T_g$ are recovered by taking $n = 1$, and the original Johnson subgroups $K_{g,1}$ and $K_g$ are recovered by taking $n=2$. 
\\\indent
The finiteness properties of the Johnson subgroups are poorly understood. For example, it is not known whether these are finitely generated for even a single $n\geq 2$. In \cite{dimcapapa12}, Dimca-Papadima showed that $H_1(K_g,\CC)$ is a finite-dimensional vector space as long as $g\geq 4$.  It is currently unknown whether $H_1(K_3,\QQ)$ is finite-dimensional. On the other hand, in \cite{akita97} Akita showed that $H_{\bullet}(K_g,\QQ)$ and $H_{\bullet}(K_{g,1},\QQ)$ are infinite-dimensional when $g\geq 7$, so $K_g$ and $K_{g,1}$ must have some infinite-rank homology in these cases. Mess has shown in \cite{mess1992} that $K_2 = T_2$ is a free group of countably infinite rank, implying that $H_1(K_2,\QQ)$ is infinite-dimensional. It is currently unknown whether $H_1(K_{g,1},\QQ)$ is finite-dimensional when $g\geq 3$.

The main goal of this note is to provide evidence that $H_1(K_{g,1},\QQ)$ is finite-dimensional. Recall that if a finite-dimensional $k$-vector space $V$ is also a module over a commutative noetherian $k$-algebra $A$, then the completion of $V$ with respect to any ideal $I\subset A$ is a finite-dimensional $k$-vector space. 
Let $H = \pi^{ab}$. Then $H_1(K_{g,1},\QQ)$ is a $\QQ H$-module\footnote{This will be shown in Section 4.}. Let $J$ denote the augmentation ideal of $\QQ H$.
\begin{theorem}\label{finite}
For each $g\geq 4$ the $J$-adic completion $H_1(K_{g,1},\QQ)^{\wedge}$ is finite-dimensional.
\end{theorem}

Theorem \ref{finite} does not automatically imply that $H_1(K_{g,1},\QQ)$ is finite-dimensional, though it is consistent with that hypothesis. 

We now provide an overview of the paper. In Section 2 we recall the basics of Torelli groups and Johnson subgroups, including the higher Johnson homomorphisms defined by Morita \cite{morita93}. In Section 3 we review a result of Papadima-Suciu \cite{papasuciu10} that will provide us with a Birman-type exact sequence for the Johnson subgroups that will be used throughout the paper. 

Section 4 is dedicated to the structure of $H_1(\pi^{(2)},\QQ)$ and its $K_{g,1}$-coinvariants. The main result of this section is that the $J$-adic filtration on $H_1(\pi^{(2)},\QQ)_{K_{g,1}}$ stabilizes. This is the main technical result needed for the proof of Theorem \ref{finite}. Along the way, we will study the graded quotients of $H_1(\pi^{(2)},\QQ)$ associated with the $J$-adic filtration, each of which has a natural $\Sp_g(\ZZ)$-module structure. Following the conventions of \cite{hain97}, let $V(\lambda)$ denote the irreducible representation of $\Sp_g(\QQ)$ with highest weight $\lambda$.
\begin{proposition}\label{proposition2}
For each $n\geq 0$ there is a natural $\Sp_g(\ZZ)$-equivariant isomorphism
\begin{equation*}
V(n\lambda_1+\lambda_2)\longrightarrow \frac{J^nH_1(\pi^{(2)},\QQ)}{J^{n+1}H_1(\pi^{(2)},\QQ)}
\end{equation*}

\end{proposition}

Section 5 is of a slightly different flavor. Here we investigate the infinite-dimensional homology of the Johnson subgroups and show that, in sufficiently large genus, the higher Johnson subgroups have infinite-dimensional homology in some degrees. 
\begin{proposition}\label{proposition1}
For each $n\geq 2$  and $g\geq 7$, the homology spaces $H_{\bullet}(K_{g,1}(n),\QQ)$ and $H_{\bullet}(K_g(n),\QQ)$ are infinite-dimensional.
\end{proposition}
 
This is a direct generalization of Akita's result. We will conclude the paper by providing a condition under which Proposition \ref{proposition1} could be sharpened considerably.
\begin{theorem}\label{theorem1}
If $H_1(\pi^{(2)},\QQ)_{K_{g,1}}$ is infinite-dimensional then for each $n\geq 2$ the direct sum
$H_2(K_g(n),\QQ)\oplus H_1(K_{g,1}(n),\QQ)$ is infinite-dimensional. 
\end{theorem}
\begin{corollary}\label{corollary1}
If $H_1(\pi^{(2)},\QQ)_{K_{g,1}}$ is infinite-dimensional, then 
for each $n\geq 2$ either $K_g(n)$ is not finitely presented or $K_{g,1}(n)$ is not finitely generated.
\end{corollary} 

\subsection{Acknowledgements}
I would like to thank Andrew Putman for many helpful discussions. Thanks are also due to the referee, whose suggestions helped to improve the exposition in this paper. 
\section{Preliminaries}\label{section2}

\subsection{Mapping Class Groups and Torelli Groups}
Let $S_{g,n}$ denote a closed orientable surface of genus $g$ with $n$ marked points, where $n = 0\ \text{or}\ 1$. The \emph{mapping class group} $\Gamma_{g,n}$ is the group of isotopy classes of orientation preserving diffeomorphisms of $S_g$ that fix the marked point. The \emph{Torelli group} $T_{g,n}$ is the kernel of the natural representation $\Gamma_{g,n}\rightarrow \Sp(H)\cong \Sp_g(\ZZ)$, where $H = H_1(S_g,\ZZ)$. We will omit the decoration $n$ when it is equal to $0$.

Johnson proved in \cite{johnson1} that $T_g$ is generated by finitely many bounding pair maps when $g\geq 3$, and this implies that $T_{g,n}$ is finitely generated for all pairs $(g,n)$ as long as $g\geq 3$. Mess \cite{mess1992} proved that $T_2$ is a free group of countably infinite rank, and showed that $T_2$ is freely generated by a set of Dehn twists on separating curves. It is currently unknown whether $T_g$ is finitely presented for even a single $g\geq 3$. %
\subsection{Johnson Subgroups}
In \cite{morita93}, Morita introduced a family of $\Gamma_{g,1}$-equivariant homomorphisms 
\begin{equation*}
\tau_{g,1}(n): K_{g,1}(n)\rightarrow \text{Hom}(H, \cLL_{n+1}(\pi))
\end{equation*}
defined by
\begin{equation}\label{johnsonhom}
\varphi\rightarrow \{\overline{x}\rightarrow \overline{\varphi(x)x^{-1}}\},
\end{equation}
where $\cLL_k(\pi) = \pi^{(k)}/\pi^{(k+1)}$. The homomorphism $\tau_{g,1}(n)$ is the \emph{$n$th Johnson homomorphism}.

Because the graded Lie algebra associated to the lower central series of $\pi$ is center-free (see Proposition \ref{prop3} below), there is a natural $\Sp_g(\ZZ)$-equivariant inclusion $\cLL_n(\pi)\hookrightarrow \text{Hom}(H, \cLL_{n+1}(\pi))$.  The \emph{$n$th outer Johnson homomorphism} $\tau_g(n)$ is a map
$$\tau_g(n): K_g(n)\rightarrow \text{Hom}(H, \cLL_{n+1}(\pi))/\cLL_n(\pi)$$
defined by a formula analogous to (\ref{johnsonhom}) above.
It can be show that the Johnson homomorphisms fit into exact sequences
\begin{equation*}
\xymatrix{
1\ar[r]&K_{g,1}(n+1)\ar[r] &K_{g,1}(n)\ar[r]^{\tau_{g,1}(n)\hspace{.4in}} & \text{Hom}(H, \cLL_{n+1}(\pi))\\
1\ar[r]&K_g(n+1)\ar[r] &K_g(n)\ar[r]^{\tau_g(n)\hspace{.5in}\ \ } & \text{Hom}(H, \cLL_{n+1}(\pi))/\cLL_n(\pi).
}
\end{equation*}

The classical Johnson homomorphism $\tau_{g,1}(1)$ was originally defined by Johnson in \cite{johnson1}. Johnson showed that the image of $\tau_{g,1}(1)$ is isomorphic to $\Lambda^3H$ and that the image of $\tau_g$ is isomorphic to $\Lambda^3H/\theta\wedge H$, where $\theta\in \Lambda^2 H$ is the symplectic form.  It is, in general, very difficult to compute the image of the higher Johnson homomorphisms\footnote{Hain \cite{hain97}, Morita \cite{morita97} and others have determined $\text{Im}(\tau_{g,1}(n))\otimes \QQ$ up to $n=6$}. The kernels $K_{g,1}$ and $K_g$ of $\tau_{g,1}(1)$ and $\tau_g(1)$, respectively, are the classical Johnson subgroups. Little is known about the structure of $K_{g,1}$ and $K_g$, aside from a handful of basic results. In \cite{johnson2}, Johnson proved the following beautiful theorem. 
\begin{theorem}[Johnson]\label{johnson}
For all $g\geq 2$, the Johnson subgroup $K_g$ is generated by Dehn twists on separating simple closed curves.
\end{theorem}

It is a difficult open problem to determine whether or not $K_g$ is finitely generated and, in light of \cite{dimcapapa12}, this cannot be determined by considering the dimension of $H_1(K_g,\QQ)$ alone, as long as $g\geq 4$. 
\section{Birman Sequence for the Johnson Subgroups}\label{section 3}
One of the key results needed throughout the paper is a Birman-type exact sequence that relates the $n$th Johnson subgroup to the $n$th outer Johnson subgroup. There is a forgetful map $T_{g,1}\rightarrow T_g$ obtained by forgetting the marked point.  This is a surjective homomorphism.
The kernel can be identified with $\pi$ via the push map $\cP: \pi\rightarrow T_{g,1}$ which sends an element $x\in \pi$ to the isotopy class of the diffeomorphism of $S_{g,1}$ obtained by dragging the marked point along $x$. For any $\gamma\in \pi$, the push map satisfies $\cP(x)(\gamma) = x\gamma x^{-1}$.
\begin{theorem}[Birman]
There is an exact sequence
\begin{equation*}
1\rightarrow \pi\xrightarrow{\cP} T_{g,1}\rightarrow T_g\rightarrow 1.
\end{equation*}
\end{theorem}

We now review some of the results from \cite{papasuciu10}. Let $G$ be a group and let $G^{(n)}$ denote the $n$th term of the lower central series of $G$. The Torelli group $\cT_G$ of a group $G$ is defined to be the kernel of the natural map $\text{Aut}(G)\rightarrow \text{Aut}(G/G^{(2)})$. The outer Torelli group $\widetilde{\cT_G}$ is defined to be the kernel of the map $\text{Out}(G)\rightarrow \text{Out}(G/G^{(2)})$. The Johnson-Andreadakis filtration $F^s\cT_G$ of $\cT_G$ is defined by
\begin{equation*}
F^s\cT_G = \text{ker}\left(\text{Aut}(G)\rightarrow \text{Aut}(G/G^{(s+1)})\right).
\end{equation*}
The Johnson-Andreadakis filtration $F^s\widetilde{\cT_G}$ of $\widetilde{\cT_G}$ is defined completely analogously.

In \cite{papasuciu10} Papadima-Suciu have worked out the relation between the terms of the filtrations on $\cT_G$ and $\widetilde{\cT_G}$, at least in certain cases. 

Recall that a group $G$ is said to be residually nilpotent if $\bigcap_{n=1}^{\infty} G^{(n)} = 1$.
\begin{theorem}[Papadima-Suciu, \cite{papasuciu10}]\label{papasuciu}
Suppose that $G$ is a residually nilpotent group whose associated graded Lie algebra $\emph{gr}\ G$ has trivial center. Then for each $n\geq 1$ there is an exact sequence
\begin{equation}\label{5}
1\rightarrow G^{(n)}\rightarrow F^n\cT_G\rightarrow F^n\widetilde{\cT_G}\rightarrow 1.
\end{equation}
\end{theorem}
The next proposition consists of standard facts which are recorded here for the reader's convenience. 
\begin{proposition}\label{prop3}
The fundamental group $\pi$ of a closed surface of genus $g\geq 2$ is residually nilpotent, and the 
graded Lie algebra associated to its lower central series has trivial center.
\end{proposition}
\begin{proof}
The first claim is a well-known result from combinatorial group theory and can be found in \cite{baumslag}. The second claim follows from the fact, proven in \cite{asadakaneko}, that the associated graded Lie algebra $\text{gr}^{\text{LCS}} \mathfrak{p}$ of the Malcev Lie algebra $\mathfrak{p}$ of $\pi$ has trivial center. This, combined with the fact the lower central 
series quotients of $\pi$ are torsion-free \cite{labute} and that
$\text{gr}^{\text{LCS}} \mathfrak{p}\cong \text{gr}\ \pi\otimes \RR$ \cite{hain97}, implies that $\text{gr}\ \pi$ has trivial center.
\end{proof}

The extended mapping class group $\Gamma^{\pm}_{g,1}$ is defined to be the group of isotopy classes of all diffeomorphisms of $S_g$ that fix the marked point (not just the orientation-preserving ones). The extended mapping class group $\Gamma_g^{\pm}$ is defined completely analogously. 
\begin{theorem}[Dehn-Nielsen-Baer, c.f. \cite{farbmargalit11}]\label{auts}
The natural maps 
\begin{equation*}
\Gamma^{\pm}_{g,1}\rightarrow \emph{Aut}(\pi)\ \ \ \ \ \ \ \ \ \Gamma^{\pm}_g\rightarrow \emph{Out}(\pi)
\end{equation*}
are isomorphisms.
\end{theorem}
Theorem \ref{auts} implies that the Torelli group and outer Torelli group of $\pi$ coincide with $T_{g,1}$ and $T_g$, respectively. 
Combined with Theorem \ref{papasuciu} we easily obtain the following.
\begin{corollary}\label{corollarydimca}
For each $n\geq 1$ there is an exact sequence
\begin{equation*}\label{extension}
1\rightarrow \pi^{(n)}\xrightarrow{\cP} K_{g,1}(n)\rightarrow K_g(n)\rightarrow 1
\end{equation*}
\end{corollary}

It is easily checked that the extension in Corollary \ref{extension} is $\pi$-equivariant with respect to the conjugation action on $\pi^{(n)}$ and $K_{g,1}(n)$ and the trivial action on $K_g(n)$.
\section{The Structure of the $K_{g,1}$-Coinvariants}
In this section we determine the $\Sp_g(\ZZ)$-module structure on the graded quotients of $H_1(\pi^{(2)},\QQ)$. The key point is that each graded quotient is irreducible. We prove this assertion by explicit computation. This fact leads to the main technical result (Corollary \ref{stab} below) concerning the space $H_1(\pi^{(2)},\QQ)_{K_{g,1}}$ of $K_{g,1}$-coinvariants needed for the proof of Theorem \ref{finite}. \subsection{The Alexander Invariant as a Module}
Let $G$ be a group. The abelian group $H_1(G^{(2)},\ZZ)$ is the Alexander invariant of $G$. It is naturally a $\ZZ G^{ab}$ module via the conjugation action of $G$. When $G$ is finitely generated,  the $\ZZ G^{ab}$-module $H_1(G^{(2)},\ZZ)$ is also finitely generated. When $G$ is finitely presented and $G^{ab}$ is torsion-free, Fox differential calculus can be used to obtain a presentation for $H_1(G^{(2)},\ZZ)$ as a $\ZZ H$-module \cite{suciu10}.

Let $F = \langle x_1,\ldots, x_r \rangle$ be a free group or rank $r$, and let $\phi: F\rightarrow H$ denote its abelianization. The Fox derivatives $D_j$, $j=1,\ldots,r$ are abelian group homomorphisms $\ZZ F\rightarrow \ZZ H$ satisfying the following properties:
\begin{enumerate}
\item $D_i(x_j) = \delta_{ij}\cdot 1$
\item $D_i(xy) = D_i(x) + \phi(x)D_i(y)$\ \ \ $x,y\in F$.
\end{enumerate}
Suppose that $\langle F\ | s_1,\ldots, s_k \rangle$ is a finite presentation for a group $G$, and assume that the abelianization of $G$ is a free abelian group $H$. Then $H_1(G^{(2)},\ZZ)$ can be realized as a $\ZZ H$-submodule of the quotient of $(\ZZ H)^{\oplus r}$ whose presentation matrix is $\left( D_j(s_i)\right)$. 
%
%
%
We use the presentation $\langle x_1,\ldots, x_{2g} | [x_1,x_{g+1}]\cdots [x_g,x_{2g}]\rangle$ of $\pi$ in order to describe $H_1(\pi^{(2)},\ZZ)$. Define $\vartheta = [x_1,x_{g+1}]\cdots [x_g,x_{2g}]$.
\begin{lemma}
The matrix $(D_j(\vartheta))$ is equal to the row vector
\begin{equation*}
\Big(1-x_{g+1}, \ldots, 1-x_{2g}, -(1-x_1),\ldots,\ -(1-x_g)\Big).
\end{equation*}
\end{lemma}
Let $e_j\in \ZZ H^{\oplus 2g}$ denote the column vector whose $j$th entry is 1 and whose other entries are 0. Define $\Theta = \sum_{j=1}^{g} (1-x_{g+j})e_j - (1-x_j)e_{g+j}\in (\ZZ H)^{\oplus 2g}$. Then $H_1(\pi^{(2)},\ZZ)$ is naturally a $\ZZ H$-submodule of $(\ZZ H)^{\oplus 2g}/(\Theta)$, where $(\Theta)$ denotes the submodule of $(\ZZ H)^{\oplus 2g}$ spanned by $\Theta$.
\begin{lemma}\label{lemma3.1}
The $\ZZ H$-module $\left(\ZZ H\right)^{\oplus 2g}/(\Theta)$
is torsion-free.
\end{lemma}
\begin{proof}
Suppose that $x = \sum_{j=1}^{2g} a_je_j$ and that the image of $x$ in $\left(\ZZ H\right)^{\oplus 2g}/(\Theta)$ is a torsion element. Then there exist non-zero $P,Q\in \ZZ H$ such that 
\begin{equation*}
P\sum_{j=1}^g a_je_j = Q \sum_{j=1}^g (1-x_{j'})e_j - (1-x_j)e_{j'}
\end{equation*}
where $j' = g+j$. This gives a system of equations
\begin{equation}\label{equation}
Pa_j = Q(1-x_{j'})\ \ \ \ \ \ \ \ Pa_{j'} = -Q(1-x_j)
\end{equation}
for each $j$. Multiplying the first by $a_{j'}$ and the second by $a_j$ leads to the equation
\begin{equation*}
Q\left(a_{j'}(1-x_{j'}) + a_j(1-x_j)\right) = 0.
\end{equation*}
Since $(1-x_j)$ and $(1-x_{j'})$ are prime ideals in the integral domain $\ZZ H$, we have $a_j\in (1-x_{j'})$ and $a_{j'}\in (1-x_j)$. Now write $a_j = r_j(1-x_{j'})$ and $a_{j'} = r_{j'}(1-x_j)$. Then by (\ref{equation}) we obtain the equations
\begin{equation}\label{equation2}
(Pr_j-Q)(1-x_{j'}) = 0\ \ \ \ \ \ \
(Pr_{j'}+Q)(1-x_j) = 0
\end{equation}
implying that $P(r_j+r_{j'}) = 0$. That is, $r_j = -r_{j'}$. From (\ref{equation2}) we also deduce that $r_j$ does not depend on $j$. That is, $r_j$ assumes a single value $r\in \ZZ H$. Finally, we are able to write
\begin{equation*}
x = r\sum_{j=1}^g (1-x_{j'})e_j - (1-x_j)e_{j'} = r\Theta.
\end{equation*}
This completes the proof.
\end{proof}
\begin{corollary}\label{lemma4.3}
The Alexander invariant of $\pi$ is a torsion-free $\ZZ H$-module.
\end{corollary}
The rational Alexander invariant $A:= H_1(\pi^{(2)},\QQ)$ is a $\QQ H$-module. The following is deduced at once from Corollary \ref{lemma4.3}.
\begin{corollary}
The rational Alexander invariant $A$ is a torsion-free $\QQ H$-module. 
\end{corollary}
Let $J$ denote the augmentation ideal in $\QQ H$ and let $\text{gr}_{\bullet} A$ denote the graded $\QQ$-vector space associated to the $J$-adic filtration on $A$.
\begin{proposition}\label{proposition3}
For each $k\geq 0$, the graded quotients $\emph{gr}_k A$ are non-zero. 
\end{proposition}
\begin{proof}
Because $\pi$ is finitely generated, $A$ is a finitely generated module over $\QQ H$. Since $\QQ H$ is noetherian, the Krull intersection theorem applies, (see \cite{eisenbud}, p.152) and because $A$ is a torsion-free $\QQ H$ module, the intersection
\begin{equation*}
\bigcap_{n=1}^{\infty}J^nA
\end{equation*}
must vanish. This implies that the $J$-adic filtration does not stabilize and therefore that the graded quotients are all non-zero.
\end{proof}
\subsection{Computing the Graded Quotients}

Following the conventions of \cite{hain97}, denote the irreducible representation of $\Sp_g(\QQ)$ with highest weight $\lambda$ by $V(\lambda)$. Define $H_{\QQ} = H\otimes \QQ$.

\begin{proposition}
There is a surjective $\Sp_g(\ZZ)$-equivariant map 
\begin{equation*}
\varphi:\emph{Sym}^n(H_{\QQ})\otimes \Lambda^2(H_{\QQ})\rightarrow \emph{gr}_{n}A
\end{equation*}
that satisfies
\begin{equation}\label{formula}
\varphi(x_1^{n_1}\cdots x_{2g}^{n_{2g}}\otimes y\wedge z) = (x_1-1)^{n_1}\cdots (x_{2g}-1)^{n_{2g}}[y,z].
\end{equation}
\end{proposition}
\begin{proof}
There is a surjective map
$$J^n/J^{n+1}\otimes_{\QQ} A_H \rightarrow \text{gr}_{n}A$$
given by multiplication which is easily seen to be $\Gamma_{g,1}$-equivariant (where $\Gamma_{g,1}$ acts diagonally on the left). Since the Torelli group $T_{g,1}$ acts trivially on both sides, this map is actually $\Sp_g(\ZZ)$-equivariant. The 5-term exact sequence for the extension $\pi^{(2)}\rightarrow \pi\rightarrow H$ produces an $\Sp_g(\ZZ)$-equivariant isomorphism $A_H\rightarrow \Lambda^2_0H$ defined by $[y,z]\rightarrow y\wedge z$. There is also an $\Sp_g(\ZZ)$-equivariant isomorphism $\text{Sym}^n(H_{\QQ})\rightarrow J^n/J^{n+1}$ defined by sending the degree $n$ monomial $x_1^{n_1}\cdots x_{2g}^{n_{2g}}\rightarrow (x_1-1)^{n_1}\cdots (x_{2g}-1)^{n_{2g}}$. Since $\Lambda^2H_{\QQ}$ decomposes as $\Lambda^2_0H_{\QQ} + \QQ\theta$, and the class of $\theta$ vanishes in $H_1(\pi^{(2)},\QQ)$, we obtain a well-defined map
$$\text{Sym}^n(H_{\QQ})\otimes \Lambda^2H_{\QQ}\rightarrow \text{gr}_{n+2}A$$
with the desired properties.
\end{proof}
Let $\text{Sym}^{\bullet}(H_{\QQ})$ denote the graded ring $\bigoplus_{n\geq 0} \text{Sym}^n(H_{\QQ})$. The tensor product 
\begin{equation*}
\text{Sym}^{\bullet}(H_{\QQ})\otimes \Lambda^2H_{\QQ} = \bigoplus_{n\geq 0} \text{Sym}^{n}(H_{\QQ})\otimes \Lambda^2H_{\QQ}
\end{equation*}
is then both a graded $\text{Sym}^{\bullet}(H_{\QQ})$-module and a $\Sp_g(\ZZ)$-module, and there is a natural surjective $\Sp_g(\ZZ)$-equivariant map
$$\text{Sym}^{\bullet}(H_{\QQ})\otimes \Lambda^2H_{\QQ}\rightarrow \text{gr}_{\bullet}A$$
of graded $\text{Sym}^{\bullet}(H_{\QQ})$-modules.
\begin{proposition}
When $n\geq 2$ the highest weight decomposition of $\emph{Sym}^n(H_{\QQ})\otimes \Lambda^2H$ is 
\small
\begin{equation*}
\left\{
\begin{array}{ll} 
V(n\lambda_1+\lambda_2) + 2V(n\lambda_1)+V((n-2)\lambda_1+\lambda_2)\ \ \ \ \ \ \ \ \ \ \ \ \ \ \ \ \ \ \ \ \ \ \ \ \ \ \ \  \ \ \displaystyle g=2\\
V(n\lambda_1+\lambda_2) + 2V(n\lambda_1)+ V((n-1)\lambda_1 + \lambda_3) + V((n-2)\lambda_1+\lambda_2) \ \ \ \  g\geq 3
 \end{array} 
 \right.
\end{equation*}
\normalsize
The highest weight decomposition of $H_{\QQ}\otimes \Lambda^2H$ is
\small
\begin{equation*}
\left\{
\begin{array}{ll} 
V(\lambda_1+\lambda_2) + 2V(\lambda_1)\ \ \ \ \ \ \ \ \ \ \ \ \ \ \ \ \ \ \ \ \ \ \ \ \ \ \ \  \ \ \displaystyle g=2\\
V(\lambda_1+\lambda_2) + 2V(\lambda_1)+ V(\lambda_3)  \ \ \ \ \ \ \ \ \ \ \ \ \ \ \ \ \ \ \ g\geq 3
 \end{array} 
 \right.
\end{equation*}
\normalsize
\end{proposition}
\begin{proof}
\noindent 
Let $i: \Lambda^3H\rightarrow H\otimes \Lambda^2H$ be the $\Sp_g(\QQ)$-equivariant map given by 
$$i(x\wedge y\wedge z) = x\otimes y\wedge z + y\otimes z\wedge x + z\otimes x\wedge y.$$
When $g=2, 3$ it is straightforward to show that the following vectors are highest weight vectors in $\text{Sym}^n(H_{\QQ})\otimes \Lambda^2H_{\QQ}$ when the formulae make sense: 
\begin{table}[h]
\begin{center}
    \begin{tabular}{| l | l | l | l |}
    \hline
   Irreducible Summand & Highest Weight Vector(s)  \\ \hline
    $V(n\lambda_1+\lambda_2)$  & $a_1^n\otimes a_1\wedge a_2$ \\ \hline
    $V(n\lambda_1)$   & $a_1^n\otimes \theta$ \ \ \ \  \ \ $a_1^{n-1}\cdot i(a_1\wedge \theta)$ \\ \hline
    $V((n-1)\lambda_1 + \lambda_3)$     & $a_1^{n-1}\cdot i(a_1\wedge a_2\wedge a_3)$ \\ \hline
    $V((n-2)\lambda_1+\lambda_2)$   & $a_1^{n-2}a_2\cdot i(a_1\wedge \theta) - a_1^{n-1}\cdot i(a_2\wedge \theta)$  \\ \hline
    \end{tabular}
    \caption{Highest weight vectors in $\text{Sym}^n(H_{\QQ})\otimes \Lambda^2H_{\QQ}$}\label{table1}
\end{center}
\end{table}
\\This shows that $\text{Sym}^n(H_{\QQ})\otimes \Lambda^2H_{\QQ}$ contains the summands claimed in the statement of the proposition. From a dimension count, using, for example, the Weyl character formula, it is readily checked that the sum of the dimensions of these irreducible summands is equal to the dimension of $\text{Sym}^n(H_{\QQ})\otimes \Lambda^2H_{\QQ}$. Thus when $g = 2$ or $3$, we have the required irreducible decomposition. To handle the cases where $g\geq 4$ we apply the stability result on p.618 of \cite{hain97}, using the fact that $\Lambda^2H_{\QQ} = V(\lambda_2)+V(0)$. 
\end{proof}
We will need the following fact, which is easily proved using the Hall-Witt identity $[x,yz] = [x,y] [x,z]^{y^{-1}}$, where $a^b = b^{-1}ab$.
\begin{lemma}
The Jacobi identity holds in $A$, i.e. for any elements $x,y,z\in \pi$ we have that
\begin{equation*}
[x,[y,z]] + [y,[z,x]] + [z,[x,y]] = 0.
\end{equation*}
\end{lemma}
\begin{proposition}
The image in $\emph{gr}_{n}A$ of each highest weight vector in Table \ref{table1} except $a_1^n\otimes a_1\wedge a_2$ vanishes.
\end{proposition}
\begin{proof}
The vanishing of the image of each highest weight vector in Table \ref{table1} besides $a_1^n\otimes a_1\wedge a_2$ follows readily from an application of the Jacobi Identity and the fact that $\theta$ maps to zero in $A$. On the other hand, $\text{gr}_{n}A\neq 0$, so the image of $a_1^n\otimes a_1\wedge a_2$ must be non-zero. 
\end{proof}
We are now in a position to give the proof of Proposition \ref{proposition2}. Recall that any finite-dimensional irreducible representation of $\Sp_g(\QQ)$ is also an irreducible representation of $\Sp_g(\ZZ)$. 
\begin{proof}[Proof of Proposition \ref{proposition2}]
For each $n\geq 0$, the $\Sp_g(\ZZ)$-equivariant map $$V(n\lambda_1+\lambda_2)\rightarrow \text{gr}_{n}A$$ is surjective. Since $\text{gr}_nA\neq 0$ and $V(n\lambda_1+\lambda_2)$ is an irreducible representation of $\Sp_g(\ZZ)$, this map must be an isomorphism by Schur's Lemma.
\end{proof}
\subsection{The Filtration Stabilizes}
Let $J_{K_{g,1}}$ denote the augmentation ideal in the group algebra $\QQ K_{g,1}$. Because $K_{g,1}$ is a normal subgroup of $\Gamma_{g,1}$,  the subspace $J_{K_{g,1}}A$ of $A$ is actually a $\QQ H$-submodule. It follows that space of coinvariants $A_{K_{g,1}}$ is naturally a $\QQ H$-module. Let $\text{gr}_{\bullet}A_{K_{g,1}}$ denote the graded $\QQ$-vector space associated with the $J$-adic filtration on $A_{K_{g,1}}$.

For each $n\geq 0$, the canonical projection $A\rightarrow A_{K_{g,1}}$ induces a surjective $\Sp_g(\ZZ)$-equivariant map
\begin{equation}\label{10}
\psi_n: \text{gr}_{n}A\rightarrow \text{gr}_{n}A_{K_{g,1}}.
\end{equation}
Since the domain is an irreducible $\Sp_g(\ZZ)$-module, this map is either zero or an isomorphism. 

To show that the $J$-adic filtration on $\text{gr}_{N}A_{K_{g,1}}$ must stabilize, we first exhibit a non-zero element of $\text{gr}_{N}A$ for some sufficiently large $N$. We then show that this particular element lies in the kernel of $\psi_N$. Since $\text{gr}_{N}A$ is irreducible, this implies that $\text{gr}_{N}A_{K_{g,1}} = 0$. 

In the calculations that follow, we will use the generating set $\{a_j,b_j\}_{j=1}^{2g}$ for $\pi$ pictured below in Figure \ref{figure} along with the separating curve $c$. 
Let $T_c$ denote a Dehn twist on the separating simple closed curve $c$. Then $T_c\in K_{g,1}$.
\begin{lemma}\label{lemmakg1}
The element $(a_1-1)(a_2-1)[a_1,b_1]\in A$ lies in $J_{K_{g,1}}A$. 
\end{lemma}
\begin{proof}
We have $T_{c}(a_1) = [[a_1,b_1],a_1]a_1$ and $T_{c}(a_2) = a_2$. By the Hall-Witt identity, then, the following identity is satisfied in $H_1(\pi^{(2)},\QQ)$:
\begin{align*}
T_{c}[a_1,a_2] &= [[[a_1,b_1],a_1]a_1, a_2]\\
&= [a_1,a_2] + [[[a_1,b_1],a_1],a_2]\\
& = [a_1,a_2] + (a_1-1)(a_2-1)[a_1,b_1].
\end{align*}
\end{proof}
\begin{figure}[h]
\centering{
\includegraphics[scale=0.5]{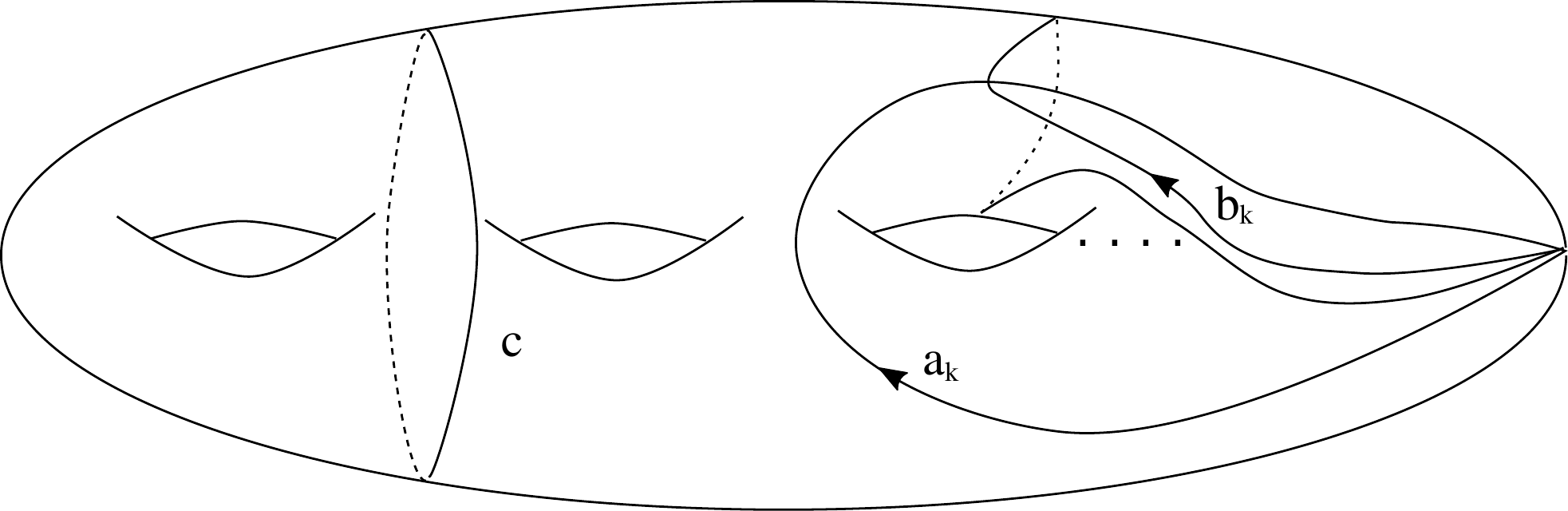}}
\caption{Generators for $\pi$ and a separating curve.}\label{figure} 
\end{figure} 

From now on, we denote the element $(a_1-1)(a_2-1)[a_1,b_1]$ by $w$. The following property of $w$ is crucial.
\begin{lemma}
The vector $w\in A$ is non-zero.
\end{lemma}
\begin{proof}
Observe first that $[a_1,b_1]\neq 0$, as it projects to the non-zero element $a_1\wedge b_1$ via the canonical map $$A\rightarrow A_H\cong \Lambda_0^2H_{\QQ},$$ where $\Lambda_0^2H_{\QQ}$ denotes the unique $\Sp_g(\QQ)$-invariant complement of the trivial representation in $\Lambda^2 H_{\QQ}$. Since $A$ is a torsion-free $\QQ H$-module, we therefore also have $(a_1-1)(a_2-1)[a_1,b_1]\neq 0$.
\end{proof}
Since $A$ is a finitely generated torsion-free $\QQ H$-module, the Krull Intersection Theorem implies the following.
\begin{lemma}\label{lemmatorsion}
For each $g\geq 2$ there exists a unique integer $N_g\geq 2$ with the property that $w$ is contained in $J^{N_g}A$ but not in $J^{{N_g}+1}A$.
\end{lemma}
Let $N$ denote the integer that is guaranteed to exist by Lemma \ref{lemmatorsion}. Then the class of $w$ in $\text{gr}_{N_g}A$ is non-zero. We claim that it lies in the kernel of the projection onto $\text{gr}_{N_g}A_{K_{g,1}}$. To see this, observe that there is a commutative diagram
\begin{equation*}
\xymatrix{
A\ar[r]^{p}&A_{K_{g,1}}\\
J^{N_g}A\ar[r]^{p'}\ar[u] & J^{N_g}A_{K_{g,1}}\ar[u]
}
\end{equation*}
where the vertical maps are inclusions and the horizontal maps are the canonical projections. Because $w\in J_{K_{g,1}}A$ by Lemma \ref{lemmakg1}, we have $p(w) = 0$. By commutativity of the diagram, we obtain $p'(w) = 0$. It follows that the class of $w$ in $\text{gr}_{N_g}A$ is mapped to $0$ in $\text{gr}_{N_g}A_{K_{g,1}}$. 

\begin{corollary}\label{stab}
The $J$-adic filtration of $A_{K_{g,1}}$ stabilizes.
\end{corollary}
\begin{proof}
Since $w$ does not vanish in $\text{gr}_{N_g}A$, it follows from Schur's Lemma that $\psi_{N_g}$ vanishes identically. This implies that $J^{N_g}A_{K_{g,1}} = J^kA_{K_{g,1}}$ for all $k\geq N_g$.
\end{proof}
\begin{corollary}\label{cor5}
The $J$-adic completion $A_{K_{g,1}}^{\wedge}$ is isomorphic, as a $\QQ H$-module, to $A/J^{N_g}A$. In particular, it is a finite-dimensional vector space.
\end{corollary}
\begin{proof}
This follows readily from Corollary \ref{stab} along with standard properties of completions.
\end{proof}
We conclude this section with the proof of Theorem \ref{finite}. By Corollary \ref{corollarydimca} there is an extension
\begin{equation}\label{ext}
1\rightarrow \pi^{(2)}\rightarrow K_{g,1}\rightarrow K_g\rightarrow 1
\end{equation}
giving rise to a 5-term exact sequence with a segment
\begin{equation}\label{5termexact8}
H_1(\pi^{(2)},\QQ)_{K_{g,1}}\rightarrow H_1(K_{g,1},\QQ)\rightarrow H_1(K_g,\QQ)\rightarrow 0.
\end{equation}
Since (\ref{ext}) is $\pi$-equivariant, 
(\ref{5termexact8}) is a sequence of $\QQ H$-modules. Note that $H$ acts trivially on $H_1(K_g,\QQ)$.
\begin{proof}[Proof of Theorem \ref{finite}]
When $g\geq 4$, $H_1(K_g,\QQ)$ is a finite-dimensional vector space and therefore a finitely generated $\QQ H$-module. Since $H_1(\pi^{(2)},\QQ)$ is a finitely generated $\QQ H$-module, so is $H_1(\pi^{(2)},\QQ)_{K_{g,1}}$. The $J$-adic completion functor is exact on the category of finitely generated $\QQ H$-modules, so we have an exact sequence
$$\left(H_1(\pi^{(2)},\QQ)_{K_{g,1}}\right)^{\wedge}\rightarrow H_1(K_{g,1},\QQ)^{\wedge}\rightarrow H_1(K_g,\QQ)\rightarrow 0.$$
Since $\left(H_1(\pi^{(2)},\QQ)_{K_{g,1}}\right)^{\wedge}$ is finite-dimensional by Corollary \ref{cor5}, $H_1(K_{g,1},\QQ)^{\wedge}$ must be finite-dimensional as well.
\end{proof}

We conclude this section with a discussion of the structure of $J_{K_{g,1}}A$. Our first observation concerns how this subspace is situated relative to the $J$-adic filtration.
\begin{lemma}
The submodule $J_{K_{g,1}}A$ is contained in $J^2A$.
\end{lemma}
\begin{proof}
Suppose that $x_1,x_2\in \pi$. By definition, any element $f\in K_{g,1}$ satisfies $f(x_j)x_j^{-1}\in \pi^{(3)}$. Let $\delta_j\in \pi^{(3)}$ be elements with the property that $f(x_j) = \delta_j x_j$. Then by the Hall-Witt identity we obtain
\begin{equation*}
f([x_1,x_2]) = [\delta_1x_1,\delta_2x_2] = [x_1,\delta_2]^{\delta_1^{-1}}[x_1,x_2]^{\delta_2^{-1}\delta_1^{-1}}[\delta_1,\delta_2][\delta_1,x_2]^{\delta_2^{-1}}.
\end{equation*}
This implies that the element $f-1\in J_{K_{g,1}}$ satisfies 
\begin{equation}\label{commutator}
(f-1)[x,y] = (x_1-1)\cdot \delta_2 - (x_2-1)\cdot \delta_1\in H_1(\pi^{(2)},\QQ).
\end{equation}
Note that the image of the natural map $\pi^{(3)}\rightarrow A$ is contained in the subspace $JA$. This observation, combined with (\ref{commutator}), implies that each $f\in K_{g,1}$ satisfies $(f-1)[x,y]\in J^2A$.
\end{proof}
We remark that the same method of proof demonstrates that any $f\in K_{g,1}(n)$ satisfies $(f-1)[x,y]\in J^nA$.
\\\\\indent 
It appears to be a very difficult problem to compute the dimension of the space of coinvariants $A_{K_{g,1}}$. Since $J^2A$ has finite codimension in $A$, this problem is equivalent to computing the codimension of $J_{K_{g,1}}A$ in $J^2A$. This problem can be solved ``infinitesimally" in the sense that we can work out the precise relationship between these subspaces after passing to completions. The following is easily deduced from Corollary \ref{cor5}.

\begin{proposition}\label{propinf}
The inclusion $J_{K_{g,1}}A\subset J^{N_g}A$ induces an isomorphism
$$ \left(J_{K_{g,1}}A\right)^{\wedge} \cong \left(J^{N_g}A\right)^{\wedge}$$
on $J$-adic completions.
\end{proposition}
\begin{question}
Does $N_g = 2$?
\end{question}

\begin{question}\label{question1}
Is there an equality $J_{K_{g,1}}A = J^{N_g}A$?
\end{question}
An affirmative answer to Question \ref{question1} would imply that $A_{K_{g,1}}$ is finite-dimensional. 

\vspace{.1in}
Since Dimca-Papadima's work \cite{dimcapapa12} implies that $H_1(K_g,\QQ)$ is finite-dimensional for $g\geq 4$, the following is deduced at once from the exact sequence (\ref{5termexact8}). 
\begin{proposition}
If $A_{K_{g,1}}$ is finite-dimensional, then $H_1(K_{g,1},\QQ)$ is finite-dimensional as long as $g\geq 4$. 
\end{proposition}

\section{An Extension of a Result of Akita}
By a direct generalization of the arguments of Akita, we will show that the higher Johnson subgroups must have infinite-dimensional homology in at least some degrees. 
\begin{lemma}[\cite{akita97}]\label{cw}
Let  $F\rightarrow E\rightarrow B$ be a fibration with $B$ a finite CW complex. If $H_{\bullet}(F,\QQ)$ is finite-dimensional, then so it $H_{\bullet}(E,\QQ)$ 
\end{lemma}
\begin{proof}[Proof of Proposition \ref{proposition1}]
For each $n\geq 2$ there is a short exact sequence 
\begin{equation*}
1\rightarrow K_{g,1}(n+1)\rightarrow K_{g,1}(n)\xrightarrow{\tau_{g,1}(n)} V_{\ZZ}\rightarrow 1
\end{equation*}
where $V_{\ZZ}\subset \text{Hom}(H, \cLL_n(\pi))$ is a free abelian group of finite rank. Thus for each $n$ there is a fibration of classifying spaces of the form
\begin{equation*}
BK_{g,1}(n+1)\rightarrow BK_{g,1}(n)\rightarrow \TT_n
\end{equation*}
where $\TT_n$ is a compact torus of dimension $\text{rank}_{\ZZ}(V_{\ZZ})$. By Akita, it is known that $H_{\bullet}(K_{g,1},\QQ)$ has infinite-dimension as long as $g\geq 7$. Appropriately modified statements hold with $K_g(n)$ in place of $K_{g,1}(n)$. An inductive argument combined with Lemma \ref{cw} finishes the proof. 
\end{proof}
We can say more in genus 2. Since $K_2 = T_2$ is a free group by \cite{mess1992}, the subgroups $K_2(n)$ are all free when $n\geq 2$. 
\begin{proposition}
For each $n\geq 2$ there is a splitting $K_{2,1}(n)\cong K_2(n)\ltimes \pi^{(n)}$.
\end{proposition}
Furthermore, since the cokernel of the inclusion $K_2(n+1)\rightarrow K_2(n)$ is a finitely generated group, we obtain the following by induction.
\begin{proposition}
For each $n\geq 2$, the rational homology groups $H_1(K_{2,1}(n),\QQ)$ and $H_1(K_2(n),\QQ)$ are infinite-dimensional.
\end{proposition}
\begin{proof}
It is clear from Mess's computation that $H_1(K_{2,1},\QQ)$ is infinite-dimensional. The key observation, now, is that for each $n$ the image of $H_1(K_{g,1}(n+1),\QQ)$ in $H_1(K_{g,1}(n),\QQ)$ has finite codimension. The result follows by induction.
\end{proof}

It turns out that if the space of coinvariants $A_{K_{g,1}}$ is infinite-dimensional, then Proposition \ref{proposition1} can be considerably sharpened. If this condition were satisfied, then for all $n\geq 2$ at least one of $H_1(K_{g,1}(n),\QQ)$ and $H_2(K_g(n),\QQ)$ would be infinite-dimensional. We will prove this in two steps.
\begin{lemma}\label{lemma3.2}
The space of coinvariants $A_{K_{g,1}}$ is infinite-dimensional if and only if the space of coinvariants $H_1(\pi^{(n)},\QQ)_{K_{g,1}}$ is infinite-dimensional for all $n\geq 3$. 
\end{lemma}
\begin{proof}
Consider the exact sequence
\begin{equation*}
H_1(\pi^{(n+1)},\QQ)_{\pi^{(n)}}\rightarrow H_1(\pi^{(n)},\QQ)\rightarrow \cLL_n(\pi)\otimes \QQ\rightarrow 0.
\end{equation*}
Because the operation of taking $K_{g,1}$-coinvariants is right-exact, there is an exact sequence
\begin{equation*}
H_1(\pi^{(n+1)},\QQ)_{K_{g,1}}\rightarrow H_1(\pi^{(n)},\QQ)_{K_{g,1}}\rightarrow \cLL_n(\pi)\otimes \QQ\rightarrow 0,
\end{equation*}
as $K_{g,1}$ acts trivially on $\cLL_n(\pi)$. Since $\cLL_n(\pi)\otimes \QQ$ is a finite-dimensional vector space, the result now follows by induction.
\end{proof}
\begin{proof}[Proof of Theorem \ref{theorem1}]
Assume that $A_{K_{g,1}}$ is infinite-dimensional. Then by Lemma \ref{lemma3.2} so is $H_1(\pi^{(n)},\QQ)_{K_{g,1}}$ for each $n\geq 2$. Since $K_{g,1}(n)$ is a subgroup of $K_{g,1}$, there is a surjection
$H_1(\pi^{(n)},\QQ)_{K_{g,1}(n)}\rightarrow H_1(\pi^{(n)},\QQ)_{K_{g,1}}$.
As the latter space is infinite-dimensional, this implies that $H_1(\pi^{(n)},\QQ)_{K_{g,1}(n)}$ is as well. 
The 5-term exact sequence of the extension (\ref{5}) has a segment
\begin{equation*}
H_2(K_g(n),\QQ)\rightarrow H_1(\pi^{(2)},\QQ)_{K_{g,1}}\rightarrow H_1(K_{g,1}(n),\QQ).
\end{equation*}
Therefore, at least one of $H_2(K_g(n),\QQ)$ and $H_1(K_{g,1}(n),\QQ)$ is infinite-dimensional. 
\end{proof}


\vspace{.15in}
\noindent Kevin Kordek\\
Department of Mathematics\\
Mailstop 3368\\
Texas A$\&$M University\\
College Station, TX 77843-3368\\
E-mail: \sf{kordek@math.tamu.edu}
 
\end{document}